\documentclass[a4paper,12pt,reqno]{amsart}
\usepackage[T1]{fontenc}
\usepackage[utf8]{inputenc}
\usepackage[libertinus,vvarbb]{newtx}
\usepackage[margin=1in]{geometry}
\usepackage{enumitem}
\usepackage{graphicx}
\usepackage[svgnames]{xcolor}
\usepackage{xparse}
\usepackage{xurl}
\usepackage[bookmarksdepth=2]{hyperref}

\hypersetup{colorlinks=true,urlcolor=MidnightBlue,linkcolor=MidnightBlue,citecolor=MidnightBlue}

\numberwithin{equation}{section}

\newtheorem{theorem}{Theorem}[section]
\newtheorem{lemma}[theorem]{Lemma}
\newtheorem{corollary}[theorem]{Corollary}
\newtheorem{proposition}[theorem]{Proposition}

\theoremstyle{definition}

\newtheorem{remark}[theorem]{Remark}

\DeclareMathOperator{\re}{Re}
\DeclareMathOperator{\im}{Im}

\DeclareMathOperator{\sign}{sign}

\newcommand{\ind}{\mathbb{1}}

\newcommand{\pr}{\mathbb{P}}
\newcommand{\ex}{\mathbb{E}}
\newcommand{\C}{\mathbb{C}}
\newcommand{\R}{\mathbb{R}}

\newcommand{\fourier}{\mathscr{F}}
\newcommand{\ph}{\varphi}

\newcommand{\dir}{\mathrm{D}}
\newcommand{\neu}{\mathrm{N}}

\mathchardef\mathhyphen="2D

\newcommand{\cm}{\mathscr{C}\mspace{-4mu}\mathscr{M}}
\newcommand{\am}{\mathscr{A}\mspace{-4mu}\mathscr{M}}
\newcommand{\amcm}{\am\mspace{-2mu}\mathhyphen\mspace{-1mu}\cm}

\renewcommand{\le}{\leqslant}
\renewcommand{\ge}{\geqslant}

\NewDocumentCommand{\formula}{ssom}{%
 \IfBooleanTF{#1}{%
  \IfBooleanTF{#2}{%
   \IfValueTF{#3}%
    {\begin{align}\label{#3}\begin{gathered}#4\end{gathered}\end{align}}%
    {\begin{gather}#4\end{gather}}%
  }{%
   \IfValueTF{#3}%
    {\begin{align}\label{#3}\begin{aligned}#4\end{aligned}\end{align}}%
    {\begin{gather*}#4\end{gather*}}%
  }%
 }{%
  \IfValueTF{#3}%
   {\begin{align}\label{#3}#4\end{align}}%
   {\begin{align*}#4\end{align*}}%
 }%
}

\begin{document}

\title{Poisson kernels on the half-plane are bell-shaped}
\author{Mateusz Kwaśnicki}
\thanks{Work supported by the National Science Centre, Poland, grant no.\@ 2023/49/B/ST1/04303}
\address{Mateusz Kwaśnicki \\ Department of Analysis and Stochastic Processes \\ Wrocław University of Science and Technology \\ ul. Wybrzeże Wyspiańskiego 27 \\ 50-370 Wrocław, Poland}
\email{\href{mailto:mateusz.kwasnicki@pwr.edu.pl}{\textsf{mateusz.kwasnicki@pwr.edu.pl}}}
%\date{\today}

\begin{abstract}
Consider a second-order elliptic operator $L$ in the half-plane $\R \times (0, \infty)$ with coefficients depending only on the second coordinate. The Poisson kernel for $L$ is used in the representation of positive $L$-harmonic functions, that is, solutions of $L u = 0$. In probabilistic terms, the Poisson kernel is the density function of the distribution of the diffusion in $\R \times (0, \infty)$ with generator $L$ at the hitting time of the boundary. We prove that the Poisson kernel for $L$ is bell-shaped: its $n$th derivative changes sign $n$ times. In particular, it is unimodal and it has two inflection points (it is concave, then convex, then concave again).
\end{abstract}

\maketitle

%
%                            ---------- o ----------
%

\section{Introduction and statement of the result}

Since the influential work of Caffarelli and Silvestre~\cite{cs07}, the \emph{harmonic extension technique} has become a standard method of tackling nonlocal problems using local PDE methods. Identification of the square root of the Laplace operator with the Dirichlet-to-Neumann operator in half-space is a classical result, and~\cite{cs07} provides a similar description of other fractional powers of the Laplace operator. While similar ideas appeared earlier, also in a probabilistic context (the link between corresponding jump-type Markov processes and diffusions goes back to the work of Molchanov and Ostrovski~\cite{mo69}; see also~\cite{cfy06,d04,gz03,m87,s58}), Caffarelli and Silvestre stated the method clearly, thereby making it readily accessible to the PDE community. They were also the first to give an explicit and surprisingly simple expression for the corresponding Poisson kernel, see~\eqref{eq:cs:poisson} below.

Certain more general nonlocal operators are known to allow a similar harmonic extension technique; that is, they are linked to solutions of appropriate PDEs in a half-space. An explicit expression for the Poisson kernel is, however, rarely available. Below we consider a class of translation-invariant nonlocal operators, for which the harmonic extension technique was developed in~\cite{k22} (see also~\cite{k23} for a probabilistic counterpart). Using the results of~\cite{ks22}, we prove a certain `geometric' property of the corresponding Poisson kernel: it is \emph{bell-shaped}, that is, its $n$th derivative changes sign exactly $n$ times for $n = 1, 2, \ldots$

We state our main result, Theorem~\ref{thm:main}, in Section~\ref{sec:main}. First, however, we give a short introduction to the harmonic extension technique in Section~\ref{sec:harmonic}, and we discuss the elliptic equations covered by our methods in Section~\ref{sec:elliptic}. Probabilistic interpretation of our work is presented in Section~\ref{sec:probability}. We mention here that a parallel development in~\cite{w} covers the case of discrete variables, and that our main theorems contain as a special case the main results of~\cite{js15,s15,y90,y92}.

\subsection{Harmonic extensions}
\label{sec:harmonic}

By the classical Poisson representation formula, given any bounded Borel function $f$ on $\R^d$, the function $u$ given by
\formula{
 u(x, y) & = \frac{\Gamma(\tfrac{d + 1}{2})}{\pi^{(d + 1) / 2}} \int_{\R^d} \frac{y}{(\lvert x - x' \rvert^2 + y^2)^{(d + 1) / 2}} \, f(x - x') dx'
}
is harmonic in the upper half-space $\R^d \times (0, \infty)$ and it converges to $f(x)$ in an appropriate sense as $y \to 0^+$. Conversely, every bounded harmonic function $u$ in the upper half-space is of the above form. We say that $u$ is the \emph{harmonic extension} of $f$, and it is now well-understood that various fine properties of $f$, such as its regularity, can be conveniently expressed in terms of $u$; see, for example, Stein's monograph~\cite{s70}. The integral kernel
\formula{
 P_{(x, y)}(x') & = \frac{\Gamma(\tfrac{d + 1}{2})}{\pi^{(d + 1) / 2}} \, \frac{y}{(\lvert x - x' \rvert^2 + y^2)^{(d + 1) / 2}}
}
is known as the \emph{Poisson kernel} for the upper half-space.

In~\cite{cs07}, Caffarelli and Silvestre generalised the above construction by considering $\alpha \in (0, 2)$ and
\formula{
 u(x, y) & = \frac{\Gamma(\tfrac{d + \alpha}{2})}{\alpha \pi^{d / 2} \Gamma(\tfrac{\alpha}{2})} \int_{\R^d} \frac{y^\alpha}{(\lvert x - x' \rvert^2 + y^2)^{(d + \alpha) / 2}} \, f(x - x') dx' .
}
They proved that $u$ is a solution of the elliptic PDE $L u = 0$, where
\formula{
 L u(x, y) & = \nabla_{x, y} \cdot (y^{1 - \alpha} \nabla_{x, y} u)(x, y) , 
}
and again $u(x, y)$ converges to $f(x)$ in an appropriate sense as $y \to 0^+$. This enabled the use of standard (local) PDE techniques in the study of a nonlocal operator, the fractional Laplace operator
\formula{
 -(-\Delta)^{\alpha / 2} f(x) & = \frac{2^\alpha \Gamma(\tfrac{d + \alpha}{2})}{\pi^{d / 2} \lvert \Gamma(-\tfrac{\alpha}{2}) \rvert} \operatorname{pv}\int_{\R^d} \frac{f(x') - f(x)}{\lvert x' - x \rvert^{d + \alpha}} \, dx' ,
}
because it coincides with the corresponding \emph{Dirichlet-to-Neumann operator}:
\formula{
 -(-\Delta)^{\alpha / 2} f(x) & = \frac{\lvert \Gamma(-\tfrac{\alpha}{2}) \rvert}{\alpha 2^\alpha \Gamma(\tfrac{\alpha}{2})} \lim_{y \to 0^+} \frac{u(x, y) - u(x, 0)}{y^\alpha} \, .
}
By analogy, the kernel function
\formula[eq:cs:poisson]{
 P_{(x, y)}(x') & = \frac{\Gamma(\tfrac{d + \alpha}{2})}{\alpha \pi^{d / 2} \Gamma(\tfrac{\alpha}{2})} \, \frac{y^\alpha}{(\lvert x - x' \rvert^2 + y^2)^{(d + \alpha) / 2}}
}
is said to be the Poisson kernel for $L$, and the method described above is now known as the Caffarelli--Silvestre extension technique. We refer to Section~10 in~\cite{g19} or Section~6 in~\cite{k19:fractional} for further discussion and references.

More generally, one can replace $L$ with an arbitrary second-order elliptic operator in the upper half-space, and use an analogous construction to study the corresponding Dirichlet-to-Neumann operator $K$. It turns out that if we require $L$ to be invariant under translations and rotations in the $x$ variable, then $-K$ is a \emph{complete Bernstein function} of $-\Delta$, and, conversely, every complete Bernstein function of $-\Delta$ is the Dirichlet-to-Neumann operator $-K$ for an essentially unique $L$ of this kind. This is a consequence of Krein's spectral theory of strings, and the details were given by Mucha and the author in~\cite{km18}; see~\cite{ah21,fls15,hl,st10} for closely related research. However, with very few exceptions, no explicit formula is known for the Poisson kernel for $L$ and, to the best knowledge of the author, it has not even been proved that the Poisson kernel is unimodal.

When $d = 1$, there is an analogous description for operators $L$ which are invariant under translations in the $x$ variable, but not necessarily symmetric. In this case the corresponding Dirichlet-to-Neumann operators $K$ have a kernel function which is completely monotone on $(0, \infty)$ and absolutely monotone on $(-\infty, 0)$, and again there is a one-to-one correspondence between nonlocal operators $K$ of this form and the corresponding elliptic operators $L$ (up to certain natural transformations). This result was given in~\cite{k22}, and its probabilistic variant is described in~\cite{k23}; the crucial ingredient is, however, due to Eckhardt and Kostenko, and de Branges; see~\cite{ek16} and the references therein. The Poisson kernel for these operators $L$ is again described only in terms of Fourier transform (see~\eqref{eq:poisson:fourier} and~\eqref{eq:poisson}), and except in a few special cases, explicit expressions are not available.

\subsection{Elliptic equations}
\label{sec:elliptic}

In the present paper we prove that in the context of the previous paragraph, the Poisson kernel $P_{(x, y)}(x')$ for the upper half-plane is a \emph{bell-shaped} function of $x'$. This means that it is unimodal (increasing and then decreasing), it has two inflection points (it is convex, then concave, then again convex), and so on: the $n$th derivative of the Poisson kernel changes sign exactly $n$ times. Before we state this result rigorously in Theorem~\ref{thm:main}, let us describe in more detail the class of operators $L$ that it covers.

Following~\cite{k22}, we consider functions $u$ defined either on the upper half-plane $\R \times (0, \infty)$, or a strip $\R \times (0, R)$. In order to treat the two cases simultaneously, we set $R = \infty$ in the former one. We consider a general second-order elliptic operator $L$ on $\R \times (0, R)$, with coefficients depending only on the second coordinate. That is,
\formula[eq:operator:general]{
 L & = a(y) \partial_{xx} + 2 b(y) \partial_{xy} + c(y) \partial_{yy} + d(y) \partial_x + e(y) \partial_y ,
}
where $a(y) \ge 0$, $c(y) \ge 0$ and $(b(y))^2 \le a(y) c(y)$. We say that $u$ is an $L$-harmonic extension of a bounded Borel function $f$ on $\R$ if $u$ is a bounded solution of $L u = 0$ in $\R \times (0, R)$, and if $u(x, y)$ converges to $f(x)$ as $y \to 0^+$. If $R$ is finite, then we additionally require that $u(x, y)$ converges to zero as $y \to R^-$.

The description of most general assumptions on the coefficients that allow one to define the notion of a solution of $L u = 0$ in a rigorous way seems problematic, and we will not attempt to do so here. Instead, we follow the approach of~\cite{k22}, where it was proved that an appropriate change of variables allows one to rewrite the elliptic equation $L u = 0$ in the following \emph{reduced form}:
\formula{
 (a(y) + (b(y))^2) \partial_{xx} u(x, y) + 2 b(y) \partial_{xy} u(x, y) + \partial_{yy} u(x, y) & = 0 ,
}
where $a(y)$ is a nonnegative, locally integrable function on $[0, R)$, and $b(y)$ is a locally square-integrable function on $[0, R)$. In fact, we may allow $a$ to be a locally finite nonnegative measure on $[0, R)$. Thus, from now on we will assume that
\formula[eq:operator]{
 L & = (a(dy) + (b(y))^2 dy) \partial_{xx} + 2 b(y) \partial_{xy} + \partial_{yy} .
}
The equation $L u = 0$ is then understood in a weak sense; we refer to Definition~2.2 and the following discussion in~\cite{k22} for details. Proposition~2.3 in~\cite{k22} asserts that for every square-integrable function $f$ on $\R$ there is a unique $L$-harmonic extension $u$, and Lemma~4.1 in~\cite{k22} states that
\formula[eq:poisson:fourier]{
 \fourier u(\xi, y) & = \fourier f(\xi) \ph_\xi(y) ,
}
where $\fourier$ denotes the Fourier transform with respect to the $x$ variable, and $\ph_\xi$ is a solution to an appropriate ODE, discussed in more detail in Section~\ref{sec:ode} below. In particular, if $\ph_\xi(y)$ (as a function of $\xi$) is the Fourier transform of a function (or perhaps a measure) $P_y$, then~\eqref{eq:poisson:fourier} can be written as
\formula[eq:poisson]{
 u(x, y) & = \int_{-\infty}^\infty f(x) P_{(x, y)}(x') dx' ,
}
where $P_{(x, y)}(x') = P_y(x - x')$ is the analogue of the classical Poisson kernel for $L$.

\subsection{Main result}
\label{sec:main}

A nonnegative smooth function $f$ on $\R$ is \emph{bell-shaped} if $f(x)$ converges to zero as $x \to \pm\infty$ and the $n$th derivative of $f$ changes sign exactly $n$ times for $n = 0, 1, 2, \ldots$\, A bounded Borel function $f$, or, more generally, a finite measure $f$, is \emph{weakly bell-shaped} if the convolution of $f$ with the Gauss--Weierstrass kernel $(4 \pi t)^{-1/2} \exp(-x^2 / (4 t))$ is bell-shaped for every $t > 0$. The notion of a bell-shaped function was introduced in 1940s, and the question whether all stable distributions are bell-shaped attracted significant attention before being fully resolved through the successive works of Gawronski~\cite{g84}, Simon~\cite{s15}, and the author~\cite{k20}. A complete characterisation of bell-shaped functions was given recently by Simon and the author in~\cite{ks22}. Note that a smooth weakly bell-shaped function is automatically bell-shaped, and weakly bell-shaped measures are in fact smooth functions everywhere except possibly a single point, where they may contain an atom.

We are now in position to state our main result. We point out that in fact it covers all operators of the form~\eqref{eq:operator:general}, as long as the equation $L u = 0$ makes sense; see the discussion above and in~\cite{k22}.

\begin{theorem}
\label{thm:main}
For every point $(x, y)$, the Poisson kernel $P_{(x, y)}(x')$ for the operator $L$ given by~\eqref{eq:operator} is a weakly bell-shaped function of $x'$. If $P_{(x, y)}(x')$ is a smooth function of $x'$, then it is bell-shaped.
\end{theorem}

Theorem~\ref{thm:main} provides a rather unexpected link between the theory of bell-shaped functions on one hand, and the harmonic extension technique (together with the underlying spectral theory) on the other one. On a superficial level there does not seem to be a clear connection between these two areas, and no intuitive explanation for validity of Theorem~\ref{thm:main} seems to be available.

Our result is motivated by its special case given by Jedidi and Simon in~\cite{js15}, who exploited a factorisation theorem proved by Yamazato in~\cite{y90}. All these results are phrased in probabilistic terms, similar to those used in Corollary~\ref{cor:main} below.

The proof of Theorem~\ref{thm:main} depends in a crucial way on the description of the class of bell-shaped functions given in~\cite{k20,ks22} and the theory of solutions of the equation $L u = 0$ developed in~\cite{k22}. Our strategy is as follows. We factorise the Fourier transform of the Poisson kernel into two terms, see~\eqref{eq:fact}. Then we identify one of them with a Pólya frequency function in Proposition~\ref{prop:ph:dir:pf}, and the other one with the Fourier transform of an absolutely monotone-then-completely monotone (or $\amcm$) function in Proposition~\ref{prop:rogers:potential}. These are the most technical steps of the proof, involving some complex-analytic techniques. With these auxiliary results at hand, the main theorem follows directly from the characterisation of bell-shaped functions given in~\cite{ks22}.

\begin{remark}
\label{rem:converse}
It is an interesting problem to describe the class of bell-shaped distributions that can be obtained as Poisson kernels in Theorem~\ref{thm:main}. This question was studied for one-sided bell-shaped functions in the context of Remark~\ref{rem:hitting} by Yamazato, who gave a somewhat inexplicit answer in~\cite{y90,y92}. A similar analysis can likely be carried out in the present more general setting, using the factorisation~\eqref{eq:fact} of the Fourier transform $\ph_\xi(y)$ of the Poisson kernel, but this would bring us beyond the scope of this paper.

The above comment motivates the following conjecture: the class of Poisson kernels described by Theorem~\ref{thm:main} is dense in the class of all bell-shaped distributions (in the sense of both pointwise convergence and convergence in total variation).
\end{remark}

\subsection{Probabilistic reformulation}
\label{sec:probability}

A probabilistic counterpart of the results of~\cite{k22} was given in~\cite{k23}: the Dirichlet-to-Neumann operator $K$ is the generator of the boundary trace process of the diffusion $(X(t), Y(t))$ generated by $L$; see~\cite{k} for further discussion. In a similar way, Theorem~\ref{thm:main} can be rephrased in probabilistic terms. In Proposition~\ref{prop:fourier} below we prove that the Poisson kernel $P_{(x, y)}(x')$ is the density function of the distribution of $X(t)$ at the hitting time $T_0$ of the boundary, defined by
\formula{
 T_0 & = \inf\{t \ge 0 : Y_t = 0\} .
}
Thus, Theorem~\ref{thm:main} almost immediately implies the following result.

\begin{corollary}
\label{cor:main}
Given an arbitrary initial value $(X(0), Y(0)) = (x, y)$, the distribution of $X(T_0)$ is weakly bell-shaped.
\end{corollary}

We mention here that the diffusion $(X(t), Y(t))$ belongs to the class of (continuous) Markov additive processes, or MAPs, with $X(t)$ being the additive part and $Y(t)$ playing the role of Markovian regulator.

Interestingly, in a parallel work~\cite{w} Wszoła proves an analogue of Corollary~\ref{cor:main} for discrete random walks.

\begin{remark}
\label{rem:hitting}
In the special case when $a(dy)$ vanishes and $b(y)$ is decreasing in~\eqref{eq:operator}, Corollary~\ref{cor:main} reduces to a known result for hitting times for one-dimension generalised diffusions. Indeed: if $B'(y) = b(y)$, then the process $\tilde{X}(t) = X(t) - B(Y(t)) + B(Y(0))$ is increasing. Denote by $T(s) = \sup\{t \ge 0 : \tilde{X}(t) \le s\}$ its (generalised, right-continuous) inverse. Then $Z(s) = Y(T(s))$ is called a \emph{generalised diffusion} (or a \emph{gap diffusion}). Furthermore, $X(\tau_0)$ is the hitting time of $0$ by $Z(s)$. Thus, Corollary~\ref{cor:main} describes the distribution of the hitting time for a one-dimensional generalised diffusion $Z(s)$. We refer to Section~5 in~\cite{k23} for further details and examples.

The special case of Corollary~\ref{cor:main} discussed above is essentially equivalent to the main theorem of~\cite{js15}, while the main ingredient of the proof, factorisation~\eqref{eq:fact} of the Fourier transform $\ph_\xi(y)$ of the Poisson kernel, is tantamount to the direct half of Theorem~1 in~\cite{y90}. Notably, \cite{b91,kw23} contain discrete analogues of these results.
\end{remark}

\subsection{Example}

As we have already mentioned, explicit expressions for the Poisson kernel $P_{(x, y)}(x')$ are extremely rare, the case~\eqref{eq:cs:poisson} studied by Caffarelli and Silvestre in~\cite{cs07} being the most notable and important exception. In fact, even the Fourier transform $\ph_\xi(y)$ of the Poisson kernel is rarely given by a closed-form expression; we refer to Section~5 in~\cite{k22} for a list of examples known to the author.

Below we extend the result of~\cite{cs07} and we give an expression for the Poisson kernel for operators $L$ which are homogeneous with respect to $(x, y)$. Following Section~5.6 in~\cite{k22}, we consider
\formula{
 L & = (p^2 + q^2) y^{2 / \mu - 2} \partial_{xx} - 2 q y^{1 / \mu - 1} \partial_{xy} + \partial_{yy} ,
}
where $p \ge 0$, $q \in \R$ and $\mu \in (0, 2)$. The corresponding Dirichlet-to-Neumann operator $K$ is a (non-symmetric) fractional derivative of order $\mu$. It was proved in~\cite{k22} that the Fourier transform $\ph_\xi(y)$ of the Poisson kernel is given by an appropriate hypergeometric function. Inversion of this Fourier transform is possible using, for example, entry~(3.2.12) in~\cite{bateman}; we choose, however, a more direct approach. By homogeneity, we expect that the Poisson kernel is given by
\formula{
 P_{(x, y)}(x') & = y^{-1 / \mu} P(y^{-1 / \mu} (x' - x))
}
for an appropriate function $P$. Since, for a fixed $x'$, $u(x, y) = P_{(x, y)}(x')$ is a solution of $L u = 0$, we obtain
\formula{
 (x^2 - 2 \mu q x + \mu^2 (p^2 + q^2)) P''(x) + ((3 + \mu) x - 4 \mu q) P'(x) + (\mu + 1) P(x) & = 0 .
}
This second order ODE has an explicit solution. Taking into account the boundary conditions, for $p > 0$ we obtain
\formula{
 P(x) & = C \exp\biggl(\frac{(1 - \mu) q}{p} \arctan \frac{x - \mu q}{\mu p}\biggr) \, \frac{1}{(\mu^2 p^2 + (x - \mu q)^2)^{(\mu + 1) / 2}} ,
}
where $C$ is a normalisation constant,
\formula{
 C & = \frac{(2 \mu p)^\mu}{2 \pi \Gamma(\mu)} \, \biggl| \Gamma\biggl(\frac{1 + \mu}{2} + \frac{(1 - \mu) q i}{2 p}\biggr) \biggr|^2 ;
}
we omit the rather technical details of the above calculation. For $p = 0$ and $\mu \ne 1$ the solution is somewhat simpler: we have
\formula{
 P(x) & = C \exp\biggl(-\frac{(1 - \mu) \mu p}{x - \mu q}\biggr) \, \frac{1}{\lvert x - \mu q \rvert^{\mu + 1}} \, \ind_{(0, \infty)}\biggl(\frac{x - \mu q}{1 - \mu}\biggr) ,
}
with the normalisation constant
\formula{
 C & = \frac{(\lvert 1 - \mu \rvert \mu p)^\mu}{\Gamma(\mu)} \, ;
}
once again we omit the details. Theorem~\ref{thm:main} implies that the above functions are bell-shaped.

\subsection{Structure of the paper}

In Section~\ref{sec:pre}, we collect properties of various classes of functions used throughout the paper. We also prove an auxiliary result about Rogers functions (Proposition~\ref{prop:rogers:potential}). The spectral problem of Eckhardt and Kostenko is discussed in Section~\ref{sec:ode}. In this part we prove further properties of solutions of the spectral ODE (Propositions~\ref{prop:ph:dir:pf} and~\ref{prop:ph:dir:bs}), which are the main technical ingredients needed for Theorem~\ref{thm:main}. With these tools in place, the proof of Theorem~\ref{thm:main} reduces to a few lines, which make up the brief Section~\ref{sec:pde}. Finally, Section~\ref{sec:prob} provides the proof of Corollary~\ref{cor:main}, which, while also very brief, requires a number of probabilistic concepts that we recall from~\cite{k23}.

%
%                            ---------- o ----------
%

\section{Preliminaries}
\label{sec:pre}

We use $x, y$ for (real) spatial variables, $\xi$ for the (real or complex) Fourier variable, $r, s, t$ for auxiliary real variables, and $a, b, c, R, \lambda, \zeta$ for parameters or coefficients. We often write, for example, `a function $f(x)$' instead of more formal `a function $f$' to emphasise that $f$ is a function of the spatial variable $x$. By $\ph_\xi(y)$ we denote the solution of the underlying ODE, where $\xi$ is the spectral parameter, and $\psi(\xi)$ denotes the corresponding Rogers function; note that $2 \psi(\xi) / \xi$ is often called the \emph{principal Weyl--Titchmarsh function} for the corresponding spectral problem. By $\nu(x)$ we denote the $\amcm$ function in the integral (Lévy--Khintchine) representation of $\psi(\xi)$ of the Lévy measure. Section~\ref{sec:prob} additionally introduces some standard probabilistic notation.

A function is \emph{smooth} if it is infinitely differentiable. All functions and measures are assumed to be Borel. By $\fourier f(\xi)$ we denote the Fourier transform of an integrable function $f(x)$,
\formula{
 \fourier f(\xi) & = \int_{-\infty}^\infty e^{-i \xi x} f(x) dx ,
}
and similarly $\fourier \mu(\xi)$ stands for the Fourier transform of a finite measure $\mu$.

In this section we briefly discuss completely and absolutely monotone functions, Pólya frequency functions and Rogers functions. We only recall properties needed later, and for a more comprehensive introduction to the theory of these objects in the present context, we refer to~\cite{k20,ks22}. A general account on completely monotone functions and other closely related classes of functions can be found in the book~\cite{ssv12} by Schilling, Song and Vondraček. Pólya frequency functions are covered in Karlin's monograph~\cite{k68}, while a broader discussion of Rogers functions can be found in the author's paper~\cite{k19:fluctuation}.

\subsection{Completely and absolutely monotone functions}

A function $f(x)$ is \emph{completely monotone} on $(0, \infty)$ if it is smooth on $(0, \infty)$ and its derivatives have alternating signs: $(-1)^n f^{(n)}(x) \ge 0$ for $x \in (0, \infty)$ and $n = 0, 1, 2, \ldots$\, By Bernstein's theorem, $f$ is completely monotone on $(0, \infty)$ if and only if $f$ is the Laplace transform of a nonnegative measure $\mu(ds)$ on $[0, \infty)$, that is,
\formula{
 f(x) & = \int_{[0, \infty)} e^{-s x} \mu(ds) ,
}
where it is assumed that the integral converges for every $x \in (0, \infty)$.

Absolutely monotone functions have all derivatives positive: $f(x)$ is \emph{absolutely monotone} on $(-\infty, 0)$ if $f$ is smooth on $(-\infty, 0)$ and $f^{(n)}(x) \ge 0$ for $x \in (0, \infty)$ and $n = 0, 1, 2, \ldots$\, Clearly, $f(x)$ is absolutely monotone on $(-\infty, 0)$ if and only if $f(-x)$ is completely monotone on $(0, \infty)$.

Finally, we say that $f(x)$ is \emph{absolutely monotone-then-completely monotone}, or $\amcm$ in short, if $f(x)$ is absolutely monotone on $(-\infty, 0)$ and completely monotone on $(0, \infty)$. In fact, in this case we allow $f$ to have an additional (nonnegative) atom at $0$. Thus, strictly speaking, we consider $\amcm$ measures $f(dx) = c \delta_0(dx) + f(x) dx$, rather than functions.

The notion of $\amcm$ functions appears naturally in the study of bell-shaped functions, and the name was introduced in~\cite{k20}. However, Lévy processes with Lévy measure having $\amcm$ density function were studied already by L.\,C.\,G.~Rogers in~\cite{r83}; see Section~\ref{sec:pre:rogers}.

\subsection{Pólya frequency functions}

We say that a function $f(x)$ is a \emph{Pólya frequency function} if for every finite increasing sequence $x_1, x_2, \ldots, x_n$ the matrix $(f(x_i - x_j) : i, j = 1, 2, \ldots, n)$ is totally positive, that is, has all minors nonnegative. We are only concerned with integrable Pólya frequency functions, which are known to have a particularly simple form: up to normalisation, they can be obtained as convolutions of a finite or infinite collection of exponential distributions, and a normal distribution. This is formally stated in terms of the Fourier transform in the following result, where, for simplicity, we agree that nonnegative measures concentrated at a single point are also Pólya frequency functions.

\begin{theorem}[Theorem~7.3.2(a) in~\cite{k68}]
\label{thm:pff}
An integrable function $f(x)$ is a Pólya frequency function if and only if
\formula{
 \fourier f(\xi) & = e^{-a \xi^2 + i b \xi + c} \prod_n \frac{e^{-i \lambda_n \xi}}{1 - i \lambda_n \xi}
}
for some uniquely determined constants $a \ge 0$, $b \in \R$, $c \in \R$, and a finite or infinite sequence of nonzero real numbers $\lambda_n$ such that $\sum_n \lambda_n^2 < \infty$.
\end{theorem}

We remark that $1 / \fourier f(\xi)$ extends to an entire function, which has only imaginary zeroes, and $1 / \fourier f(\xi)$ can be approximated by polynomials with the same property; that is, $1 / \fourier f(i \xi)$ is a function in the \emph{Pólya--Laguerre class}. For a detailed treatment of Pólya frequency functions, we refer to Chapter~7 in~\cite{k68}.

\subsection{Rogers functions}
\label{sec:pre:rogers}

We say that a function $\psi(\xi)$ holomorphic in the right complex half-plane $\re \xi > 0$ is a \emph{Rogers function} if $\re (\psi(\xi) / \xi) \ge 0$ for every $\xi \in \C$ such that $\re \xi > 0$. More generally, we say that a function defined initially on $(0, \infty)$ is a Rogers function if it extends to a holomorphic function with the above property, and we use the same symbol for the holomorphic extension.

The name \emph{Rogers function} was introduced by the author in~\cite{k19:fluctuation}, but the concept goes back to the work of Rogers~\cite{r83}. The following result asserts that Rogers functions are characteristic exponents of certain Lévy processes; namely, those with Lévy measure having an $\amcm$ density function. This class is called \emph{Lévy processes with completely monotone jumps} in~\cite{k19:fluctuation}.

\begin{proposition}[Theorem~3.3(a--b) and equation~(2.2) in~\cite{k19:fluctuation}]
\label{prop:rogers:levy}
Every Rogers function $\psi(\xi)$ is given by
\formula[eq:rogers:levy]{
 \psi(\xi) & = a \xi^2 - i b \xi + c + \int_{-\infty}^\infty (1 - e^{i \xi x} + i \xi (1 - e^{-\lvert x \rvert}) \sign x) \nu(x) dx
}
for some uniquely determined constants $a \ge 0$, $b \in \R$, $c \ge 0$ and a unique $\amcm$ function $\nu(x)$ which satisfies the integrability condition $\int_{-\infty}^\infty \min\{1, x^2\} \nu(x) dx < \infty$.
\end{proposition}

For other integral representations of Rogers functions, we refer to the full statement of Theorem~3.3 in~\cite{k19:fluctuation}. Among various properties of Rogers functions given in that paper, we will only need the following one.

\begin{proposition}[Proposition~3.12(b) in~\cite{k19:fluctuation}]
\label{prop:rogers:dual}
If $\psi(\xi)$ is a Rogers function, then either $\psi(\xi)$ is constant zero or $\xi^2 / \psi(\xi)$ is a Rogers function.
\end{proposition}

From~\eqref{eq:rogers:levy} it follows that for every Rogers function $\psi(\xi)$ a finite, nonnegative limit $\psi(0^+) = c$ exists; see Proposition~3.14 in~\cite{k19:fluctuation}. We now prove the following auxiliary result, which describes the case when the limit $\psi(0^+)$ is nonzero.

\begin{proposition}
\label{prop:rogers:potential}
If $\psi(\xi)$ is a Rogers function with $\psi(0^+) > 0$, then $1 / \psi(\xi)$ is the Fourier transform of an integrable $\amcm$ function.
\end{proposition}

In probabilistic terms, the above result states that the resolvent (or $\lambda$-potential) kernel of a Lévy process with completely monotone jumps is an $\amcm$ function. The converse is also true, with a very similar proof, but it will not be needed here.

\begin{proof}
\emph{Step 1.} Suppose that $\psi(\xi)$ is a Rogers function which is not constant zero. By Proposition~\ref{prop:rogers:dual}, $\xi^2 / \psi(\xi)$ is a Rogers function. Hence, by Proposition~\ref{prop:rogers:levy},
\formula{
 \frac{\xi^2}{\psi(\xi)} & = a \xi^2 - i b \xi + c + \int_{-\infty}^\infty (1 - e^{i \xi x} + i \xi (1 - e^{-\lvert x \rvert}) \sign x) \nu(x) dx
}
for some $a \ge 0$, $b \in \R$, $c \ge 0$ and an $\amcm$ function $\nu(x)$. It follows that
\formula{
 \frac{1}{\psi(\xi)} & = a + \frac{-i b \xi + c}{\xi^2} + \int_{-\infty}^\infty \frac{1 - e^{i \xi x} + i \xi (1 - e^{-\lvert x \rvert}) \sign x}{\xi^2} \, \nu(x) dx .
}
If $\xi > 0$, then
\formula{
 \re \frac{1}{\psi(\xi)} & = a + \frac{c}{\xi^2} + \int_{-\infty}^\infty \frac{1 - \cos(\xi x)}{\xi^2} \, \nu(x) dx .
}
Note that the integrand is nonnegative. Since the left-hand side has a finite limit as $\xi \to 0^+$, it follows that $c = 0$ and, by Fatou's lemma,
\formula[eq:rogers:square]{
 \int_{-\infty}^\infty \frac{x^2}{2} \, \nu(x) dx & \le \liminf_{\xi \to \infty} \int_{-\infty}^\infty \frac{1 - \cos(\xi x)}{\xi^2} \, \nu(x) dx < \infty .
}
Thus, the expression for $1 / \psi(\xi)$ becomes
\formula{
 \frac{1}{\psi(\xi)} & = a + \int_{-\infty}^\infty \frac{1 - e^{i \xi x} + i \xi x}{\xi^2} \, \nu(x) dx \\
 & \qquad + \frac{i}{\xi} \biggl(\int_{-\infty}^\infty (1 - \lvert x \rvert - e^{-\lvert x \rvert}) \sign x \, \nu(x) dx - b\biggr) ;
}
note that both integrals are absolutely convergent by~\eqref{eq:rogers:square}. Again we consider $\xi > 0$, and this time we investigate
\formula{
 \im \frac{1}{\psi(\xi)} & = \int_{-\infty}^\infty \frac{-\sin(\xi x) + \xi x}{\xi^2} \, \nu(x) dx \\
 & \qquad + \frac{1}{\xi} \biggl(\int_{-\infty}^\infty (1 - \lvert x \rvert - e^{-\lvert x \rvert}) \sign x \, \nu(x) dx - b\biggr) .
}
By the dominated convergence theorem, the first integral has a finite limit as $\xi \to 0^+$, and so we necessarily have
\formula{
 \int_{-\infty}^\infty (1 - \lvert x \rvert - e^{-\lvert x \rvert}) \sign x \, \nu(x) dx - b & = 0 .
}
We conclude that
\formula{
 \frac{1}{\psi(\xi)} & = a + \int_{-\infty}^\infty \frac{1 - e^{i \xi x} + i \xi x}{\xi^2} \, \nu(x) dx .
}

\emph{Step 2.} Define
\formula{
 v(x) & = \int_{-\infty}^x (x - s) \nu(s) ds && \text{if } x < 0 , \\
 v(x) & = \int_x^\infty (s - x) \nu(s) ds && \text{if } x > 0 .
}
Then $v''(x) = \nu(x)$ for $x \in \R \setminus \{0\}$, and so $v$ is $\amcm$. Additionally, by Fubini's theorem,
\formula{
 \int_{-\infty}^\infty v(x) dx = \int_{-\infty}^\infty \frac{s^2}{2} \, \nu(s) ds < \infty ,
}
that is, $v$ is integrable. Finally, again by Fubini,
\formula{
 \fourier v(-\xi) & = \int_{-\infty}^\infty v(x) e^{i \xi x} dx \\
 & = \int_{-\infty}^0 \int_s^0 (x - s) \nu(s) e^{i \xi x} dx ds + \int_0^\infty \int_0^s (s - x) \nu(s) e^{i \xi x} dx ds \\
 & = \int_{-\infty}^\infty \frac{1 - e^{i \xi x} + i \xi x}{\xi^2} \, \nu(s) ds .
}
Therefore, $1 / \psi(\xi)$ is the Fourier transform of an integrable $\amcm$ function $a \delta_0(dx) + v(-x) dx$, as desired.
\end{proof}

\subsection{Bell-shaped functions}

Recall that a smooth function $f(x)$ is \emph{bell-shaped} if it is nonnegative and its $n$th derivative $f^{(n)}$ changes sign exactly $n$ times for $n = 0, 1, 2, \ldots$\,, while a nonnegative function (or, more generally, a nonnegative measure) $f(x)$ is \emph{weakly bell-shaped} if the convolution of $f$ with the Gauss--Weierstrass kernel (that is, the density function of the normal distribution) is bell-shaped. We will need the following fundamental result about bell-shaped functions.

\begin{theorem}[Corollary~1.9 in~\cite{ks22}]
\label{thm:bell}
A function (or a measure) $f$ is weakly bell-shaped if and only if it is a convolution of an integrable Pólya frequency function and a locally integrable $\amcm$ function (or measure) which converges to $0$ at $\pm \infty$. If $f$ is additionally smooth, then $f$ is bell-shaped.
\end{theorem}

In fact, we only need the direct half of the above result, given already in~\cite{k20}; see Theorem~3.7, Lemma~5.4 and the proof of Theorem~1.1 therein. We remark that the main result of~\cite{ks22} additionally characterises the class of (weakly) bell-shaped functions in terms of the Fourier transform, see Theorem~1.3 therein; we will, however, not need this result here. For further properties and examples of bell-shaped functions, we refer to~\cite{k20,ks22}. Their discrete analogues, that is, bell-shaped sequences, are discussed in~\cite{kw23,kw,w}.

%
%                            ---------- o ----------
%

\section{Auxiliary ODE}
\label{sec:ode}

\subsection{Spectral problem of Eckhardt and Kostenko}

The following theorem was originally proved by Eckhardt and Kostenko in~\cite{ek16} (using de Branges's theory) in a slightly different form. We quote the variant given in~\cite{k22,k23}. More precisely, we follow closely the notation of~\cite{k23}, which differs from that of~\cite{k22} in three aspects. First, our function $\psi(\xi)$ is given by $\psi(\xi) = \tfrac{1}{2} k(\xi)$ with the notation of~\cite{k22}. Next, in~\cite{k22} the symbol $a(dy)$ denotes what is here $a(dy) + (b(y))^2 dy$. Finally, we use the unambiguous expression $\ph_\xi'(0^+) + a(\{0\}) \xi^2$ instead of $\ph_\xi'(0)$. For further discussion, we refer to the comment preceding Theorem~4.1 in~\cite{k23} and to the third paragraph in Section~2.4 in~\cite{k22}.

\begin{theorem}[Theorem~4.1 in~\cite{k23}; see also Theorem~3.1 in~\cite{k22}]
\label{thm:ode}
\mbox{}
\begin{enumerate}[label={\textnormal{(\alph*)}}]
\item\label{thm:ode:a}
Suppose that $R \in (0, \infty]$, $a(dy)$ is a locally finite nonnegative measure on $[0, R)$, and $b(y)$ is a locally square-integrable function on $[0, R)$. For every $\xi > 0$ there is a unique function $\ph_\xi(y)$ with the following properties:
\begin{itemize}
\item $\ph_\xi(y)$ is locally absolutely continuous on $[0, R)$;
\item $\ph_\xi'(y)$ is equal almost everywhere in $[0, R)$ to a function with locally bounded variation on $[0, R)$, so that the distributional derivative $\ph_\xi''$ corresponds to a locally finite measure in $[0, R)$;
\item $\ph_\xi(0) = 1$, $\ph_\xi(y)$ is bounded on $[0, R)$, and if $R \in (0, \infty)$, then $\ph_\xi(R^-) = 0$;
\item we have
\formula[eq:ode]{
 \frac{1}{2} \ph_\xi''(dy) & = \frac{1}{2} \, \xi^2 \ph_\xi(y) a(dy) + \frac{1}{2} \, \xi^2 (b(y))^2 \ph_\xi(y) dy - i \xi \ph_\xi'(y) b(y) dy
}
in $(0, R)$, in the sense of distributions.
\end{itemize}
Additionally,
\begin{itemize}
\item $\lvert \ph_\xi(y) \rvert^2$ is decreasing and convex on $[0, R)$;
\item $\lvert \ph_\xi'(y) \rvert$ is decreasing on $[0, R)$;
\item $\lvert \ph_\xi(y) b(y) \rvert^2 + \lvert \ph_\xi'(y) \rvert^2$ is integrable over $[0, R)$.
\end{itemize}
\item\label{thm:ode:b}
The function $\psi(\xi)$, defined on $(0, \infty)$ by the formula
\formula[eq:psi]{
 \psi(\xi) & = -\frac{1}{2} \, \ph_\xi'(0^+) + \frac{1}{2} \, a(\{0\}) \xi^2 ,
}
extends to a Rogers function.
\item\label{thm:ode:c}
Every Rogers function $\psi(\xi)$ can be represented as above in a unique way.
\end{enumerate}
\end{theorem}

The proof of Theorem~\ref{thm:ode} involves the following intermediate results that we will need below. For every $\xi \in \C$, the two-dimensional space of solutions of~\eqref{eq:ode} is spanned by two functions $\ph_\xi^\dir$ and $\ph_\xi^\neu$, which satisfy the Dirichlet and Neumann initial conditions, respectively:
\formula{
 \ph_\xi^\dir(0) & = 0 , & \ph_\xi^\neu(0) & = 1 , \\
 (\ph_\xi^\dir)'(0^+) & = 1 , & (\ph_\xi^\neu)'(0^+) & = \frac{1}{2} \, a(\{0\}) \xi^2 ;
}
see Lemma~A.1 in~\cite{k22}, and note the changes in the notation mentioned before the statement of Theorem~\ref{thm:ode}. We call $\ph_\xi^\dir$ and $\ph_\xi^\neu$ fundamental solutions of the ODE~\eqref{eq:ode}. Additionally, for $\xi > 0$ we have
\formula[eq:ph:dir:neu]{
 \ph_\xi(y) & = \ph_\xi^\neu(y) - 2 \psi(\xi) \ph_\xi^\dir(y)
}
and
\formula[eq:psi:dir:neu]{
 \psi(\xi) & = \lim_{y \to R^-} \frac{1}{2} \, \frac{\ph_\xi^\neu(y)}{\ph_\xi^\dir(y)} \, ;
}
see the proof of Lemma~A.6 in~\cite{k22}, and keep in mind that the function $k(\xi)$ from~\cite{k22} is denoted here by $2 \psi(\xi)$. Finally, for every fixed $y \in [0, R)$, the mappings $\xi \mapsto \ph_\xi^\dir(y)$ and $\xi \mapsto \ph_\xi^\neu(y)$ are entire functions; see Lemma~A.1 in~\cite{k22}.

\subsection{The fundamental solution}

We need more information about the fundamental solution $\ph_\xi^\dir(y)$. Virtually all results in this section remain true for the other fundamental solution $\ph_\xi^\neu$, but we will not need them.

\begin{lemma}
\label{lem:ph:dir:nonzero}
We have
\formula{
 \ph_\xi^\dir(y) & \ne 0
}
when $\re \xi > 0$ and $y \in (0, R)$.
\end{lemma}

\begin{proof}
By Lemma~A.2 in~\cite{k22}, if $B(y) = \int_0^y b(s) ds$, then
\formula[eq:ph:dir:monotone]{
 e^{-2 B(y_1) \im \xi} \re(\overline{\xi \ph_\xi^\dir(y_1)} (\ph_\xi^\dir)'(y_1^+)) & \le e^{-2 B(y_2) \im \xi} \re(\overline{\xi \ph_\xi^\dir(y_2)} (\ph_\xi^\dir)'(y_2^+))
}
for every $\xi \in \C$ with $\re \xi > 0$ and every $y_1, y_2 \in [0, R)$ with $y_1 < y_2$ (and in fact the same is true for every solution of~\eqref{eq:ode}, not just $\ph_\xi^\dir(y)$). Observe that since $\ph_\xi^\dir(0) = 0$ and $(\ph_\xi^\dir)'(0^+) = 1$, we have
\formula{
 \ph_\xi^\dir(y) & = y + o(y) , \\
 (\ph_\xi^\dir)'(y^+) & = 1 + o(1)
}
as $y \to 0^+$, and hence
\formula{
 \overline{\xi \ph_\xi^\dir(y)} (\ph_\xi^\dir)'(y^+) = \overline{\xi} y + o(y) .
}
In particular, in some right neighbourhood of $0$ we have
\formula{
 \re(\overline{\xi \ph_\xi^\dir(y)} (\ph_\xi^\dir)'(y^+)) & > 0 ,
}
and therefore, by~\eqref{eq:ph:dir:monotone}, the above inequality holds for every $y \in (0, R)$. In particular, the desired result $\ph_\xi^\dir(y) \ne 0$ follows.
\end{proof}

\begin{lemma}
\label{lem:ph:dir:order}
For every $y \in (0, R)$ there is a constant $C > 0$ such that
\formula{
 \lvert \ph_\xi^\dir(y) \rvert & \le \exp(C (1 + \lvert \xi \rvert^2))
}
for all $\xi \in \C$.
\end{lemma}

\begin{proof}
This estimate follows from the proof of Lemma~A.1 in~\cite{k22}. Indeed: $\ph_\xi^\dir$ corresponds to $\alpha = 0$, $\beta = 1$ in that proof, and we choose $C = \lvert \xi \rvert$ (line~10 on page~33 in~\cite{k22}), so that
\formula{
 M(y) & = \exp \biggl( 2 y + 4 \lvert \xi \rvert^2 a([0, y)) + 4 \lvert \xi \rvert^2 \int_0^y (b(s))^2 ds + 8 \lvert \xi \rvert \int_0^y \lvert b(s) \rvert ds \biggr)
}
(line~12 on page~33 in~\cite{k22}). Since
\formula{
 \lvert \ph_\xi^\dir(y) \rvert & \le \lVert \ph_\xi^\dir \rVert_X M(y)
}
(lines~21--22 on page~33 in~\cite{k22}) and
\formula{
 \lVert \ph_\xi^\dir \rVert_X & \le 2 \lvert \alpha \rvert + 2 \lvert \beta \rvert = 2
}
(line~2 on page~34 in~\cite{k22}), we conclude that $\lvert \ph_\xi^\dir(y) \rvert \le 2 M(y)$. The desired estimate follows.
\end{proof}

\begin{proposition}
\label{prop:ph:dir:pf}
For a fixed $y \in [0, R)$, the function $\xi \mapsto 1 / \ph_\xi^\dir(y)$ is the Fourier transform of an integrable Pólya frequency function.
\end{proposition}

\begin{proof}
Denote $\eta(\xi) = \ph_\xi^\dir(y)$. We already know that $\eta(\xi)$ is an entire function, and since the complex conjugate of $\ph_\xi^\dir(y)$ is the solution of the ODE~\eqref{eq:ode} with $\xi$ replaced by $-\bar \xi$, we have $\eta(-\bar \xi) = \overline{\eta(\xi)}$. Additionally, $\eta(0) = \ph_0^\dir(y) = y$. By Lemma~\ref{lem:ph:dir:nonzero}, $\eta(\xi)$ has no zeroes in the right complex half-plane $\re \xi > 0$, and due to $\eta(-\bar \xi) = \overline{\eta(\xi)}$, $\eta(\xi)$ has no zeroes in the left complex half-plane $\re \xi < 0$. Thus, all zeroes of $\eta(\xi)$ lie on the imaginary axis $i \R$. Let $i \zeta_n$ denote the (finite or infinite) sequence of all zeros of $\eta(\xi)$, repeated according to their multiplicity.

By Lemma~\ref{lem:ph:dir:order}, the entire function $\eta$ is of order at most $2$, and if it is of order $2$, then it is of finite type. Let $p$ be the genus of $\eta(\xi)$, that is, the least nonnegative integer such that $\sum_n \lvert \zeta_n \rvert^{-p - 1}$ is finite. By Hadamard's factorisation theorem (Theorem~1 in Lecture~4 in~\cite{l97}), we have $p \le 2$. We claim that in fact $p \le 1$.

Suppose, contrary to our claim, that $p = 2$. This is only possible when $\eta(\xi)$ is of order $2$, and in this case, since $\eta(\xi)$ is of finite type, partial sums
\formula{
 \sum_n \frac{1}{(i \zeta_n)^2} \, \ind_{[0, r)}(\lvert \zeta_n \rvert) = -\sum_n \frac{1}{\lvert \zeta_n \rvert^2} \, \ind_{[0, r)}(\lvert \zeta_n \rvert)
}
are bounded as $r \to \infty$ (see Theorem~4 in Lecture~5 in~\cite{l97}). By the monotone convergence theorem, the series $\sum_n \lvert \zeta_n \rvert^{-2}$ necessarily converges (it is here that we use the fact that all zeroes of $\eta(\xi)$ lie on the imaginary axis), and so $p \le 1$, contrary to our assumption $p = 2$. This completes the proof of our claim: we have $p \le 1$.

By Hadamard's factorisation theorem (Theorem~1 in Lecture~4 in~\cite{l97}), there are constants $a, b \in \C$ such that
\formula{
 \eta(\xi) & = y \, e^{a \xi^2 + i b \xi} \prod_n e^{\xi / (i \zeta_n)} \biggl( 1 - \frac{\xi}{i \zeta_n} \biggr) .
}
By Theorem~\ref{thm:pff}, in order to prove that $1 / \eta(\xi)$ is the Fourier transform of an integrable Pólya frequency function, it remains to show that $a \ge 0$ and $b \in \R$.

Since $\eta(-\bar \xi) = \overline{\eta(\xi)}$, we have $a \bar \xi^2 - i b \bar \xi = \overline{a \xi^2 + i b \xi}$, at least for $\xi$ sufficiently close to $0$. Therefore, $a$ and $b$ are necessarily real.

Suppose that $\xi > 0$. By Lemma~A.3 in~\cite{k22}, $\lvert \ph_\xi^\dir(y) \rvert^2$ is a convex function of $y \in [0, R)$, and since $\ph_\xi^\dir(0) = 0$ and $(\ph_\xi^\dir)'(0^+) \ge 1$, we have $\lvert \eta(\xi) \rvert^2 = \lvert \ph_\xi^\dir(y) \rvert^2 \ge 2 y$. In particular, $\eta(\xi)$ does not converge to $0$ as $\xi \to \infty$. This proves that $a \ge 0$: otherwise, by Borel's theorem (see Theorem~3 in Lecture~4 in~\cite{l97}), the term $e^{a \xi^2 + i b \xi}$ would converge to zero fast enough to suppress the growth of the Weierstrass product $\prod_n e^{\xi / (i \zeta_n)} (1 - \xi / (i \zeta_n))$.
\end{proof}

\subsection{Two subproblems}

Fix $\check R \in (0, R)$. We split the ODE~\eqref{eq:ode} into two problems. The equation on $(\check R, R)$, translated to $(0, R - \check R)$, will correspond to symbols with a hat. Parameters and solutions of the equation on $[0, \check R]$, reflected about $\tfrac{1}{2} \check R$, will be denoted with a check.

We define
\formula{
 \hat R & = R - \check R , \\
 \hat a(dy) & = \ind_{(0, \hat R)}(y) a(\check R + dy) , \\
 \hat b(y) & = b(\check R + y) ;
}
note that if $R = \infty$, then also $\hat R = \infty$. Clearly,
\formula{
 \hat \ph_\xi(y) & = \frac{\ph_\xi(\check R + y)}{\ph_\xi(\check R)}
}
is the (unique) solution of the ODE~\eqref{eq:ode} described in Theorem~\ref{thm:ode}, with $R, a, b$ replaced by $\hat R, \hat a, \hat b$. By that theorem,
\formula{
 \hat \psi(\xi) & = -\frac{1}{2} \, \hat \ph_\xi'(0^+) = -\frac{1}{2} \, \frac{\ph_\xi'(\check R^+)}{\ph_\xi(\check R)}
}
is a Rogers function.

Similarly, we define
\formula{
 \check a(dy) & = \ind_{[0, \check R)}(y) a(\check R - dy) , \\
 \check b(y) & = -b(\check R - y) ,
}
and we denote by $\check \ph_\xi^\dir$ and $\check \ph_\xi^\neu$ the fundamental solutions of the ODE~\eqref{eq:ode} with $R, a, b$ replaced by $\check R, \check a, \check b$. It is straightforward to verify that $\check \ph_\xi^\dir(\check R - y)$ and $\check \ph_\xi^\neu(\check R - y)$ solve~\eqref{eq:ode} on $(0, \check R)$. Furthermore,
\formula[eq:fact:0]{
 \ph_\xi(y) & = \ph_\xi(\check R) \check \ph_\xi^\neu(\check R - y) - \ph_\xi'(\check R^+) \check \ph_\xi^\dir(\check R - y) ,
}
and a similar equation holds true with $\ph_\xi$ replaced by $\ph_\xi^\dir$ or by $\ph_\xi^\neu$. In particular, $\check \ph_\xi^\dir$ and $\check \ph_\xi^\neu$ are continuous functions on $[0, \check R]$, and by~\eqref{eq:psi:dir:neu}, the corresponding Rogers function $\check \psi(\xi)$ is given by
\formula{
 \check \psi(\xi) & = \frac{1}{2} \, \frac{\check \ph_\xi^\neu(R)}{\check \ph_\xi^\dir(\check R)} \, .
}
In particular, $\check \psi(0) = 1 / (2 \check R) > 0$.

\subsection{Factorisation}

Identity~\eqref{eq:fact:0} with $y = 0$ reads
\formula{
 1 & = \ph_\xi(\check R) \check \ph_\xi^\neu(\check R) - \ph_\xi'(\check R^+) \check \ph_\xi^\dir(\check R) .
}
Since $\ph_\xi'(\check R^+) = -2 \hat \psi(\xi) \ph_\xi(\check R)$, we find that
\formula{
 1 & = \ph_\xi(\check R) (\check \ph_\xi^\neu(\check R) + 2 \hat \psi(\xi) \check \ph_\xi^\dir(\check R)) .
}
Furthermore, $\check \ph_\xi^\neu(\check R) = 2 \check \psi(\xi) \check \ph_\xi^\dir(\check R)$. Hence,
\formula{
 1 & = \ph_\xi(\check R) \check \ph_\xi^\dir(\check R) (2 \check \psi(\xi) + 2 \hat \psi(\xi)) .
}
This leads us to the factorisation
\formula[eq:fact]{
 \ph_\xi(\check R) & = \frac{1}{\check \ph_\xi^\dir(\check R)} \, \frac{1}{2 \check \psi(\xi) + 2 \hat \psi(\xi)} \, .
}
By Proposition~\ref{prop:ph:dir:pf}, the factor $1 / \check \ph_\xi^\dir(\check R)$ (as a function of $\xi$) is the Fourier transform of an integrable Pólya frequency function. On the other hand, $\check \psi(0) > 0$, and so, by Proposition~\ref{prop:rogers:potential}, the factor $1 / (2 \check \psi(\xi) + 2 \hat \psi(\xi))$ is the Fourier transform of an integrable $\amcm$ function. Hence, by~\eqref{eq:fact} and Theorem~\ref{thm:bell}, $\ph_\xi(\check R)$ (as a function of $\xi$) is the Fourier transform of an (integrable) bell-shaped function. We state this result below.

\begin{proposition}
\label{prop:ph:dir:bs}
For every $y \in (0, R)$, the function $\ph_\xi(y)$ is the Fourier transform of an integrable bell-shaped function.
\end{proposition}

%
%                            ---------- o ----------
%

\section{PDE interpretation}
\label{sec:pde}

Recall that we consider elliptic operators $L$ given by~\eqref{eq:operator:general}. With Proposition~\ref{prop:ph:dir:bs} at hand, the proof of Theorem~\ref{thm:main} reduces to an application of Lemma~4.1 in~\cite{k22}. This result asserts that an $L$-harmonic function $u(x, y)$ with boundary values $f(x)$ (see Definition~2.2 in~\cite{k22} for a precise meaning) satisfies
\formula{
 \fourier u(\xi, y) & = \ph_\xi(y) \fourier f(\xi)
}
for almost every $\xi \in \R$, with any fixed $y \in [0, R)$. By Proposition~\ref{prop:ph:dir:bs}, for a given $y \in [0, R)$ the function $\ph_\xi(y)$ is the Fourier transform of an integrable weakly bell-shaped function (or, more generally, a measure) $P_y(x)$. By the exchange formula,
\formula{
 u(x, y) & = \int_{\R} f(s) P_y(x - s) ds
}
for almost every $x \in \R$. Therefore, $P_{(x, y)}(s) = P_y(x - s)$ is indeed the Poisson kernel for the operator $L$, and it is weakly bell-shaped.

%
%                            ---------- o ----------
%

\section{Probabilistic interpretation}
\label{sec:prob}

The proof of Corollary~\ref{cor:main} is slightly more involved, and we adopt the notation used in~\cite{k23}. By the results of Section~4.1 in~\cite{k23}, with no loss of generality we restrict our attention to regular shift-invariant diffusions $(X(t), Y(t))$ in $\R \times [0, R)$, where $R \in (0, \infty]$. That is to say, we assume that $X(t)$ is a local martingale, and $Y(t)$ is the Brownian motion in $[0, R)$, reflected at $0$ and killed at $R$ (if $R$ is finite). We denote by $T_R$ the lifetime of $Y(t)$, and we let $T_0$ be the hitting time of $\{0\}$ for $Y(t)$.

The horizontal coordinate, $X(t)$, is driven by some Brownian motion $W(t)$ in $\R$, independent of $Y(t)$. More precisely, we have
\formula[eq:xt]{
 X(t) & = W(A(t)) + B(t) ,
}
where $A(t)$ and $B(t)$ are appropriate additive functionals of $Y(t)$:
\formula{
 A(t) & = \int_{[0, R)} L_y(t) a(dy) , \\
 B(t) & = \int_0^t b(Y(s)) d\dot Y(s) .
}
Here $\dot Y(t) = Y(t) - \tfrac{1}{2} L_0(t)$ denotes the martingale part in the canonical decomposition of the semimartingale $Y(t)$, $a(dy)$ is a nonnegative Radon measure on $[0, R)$, and $b(y)$ is a locally square-integrable real-valued function on $[0, R)$. We refer to~\cite{k23} for the details.

For $y \ge 0$ we write $\pr^y$ and $\ex^y$ for the probability and expectation corresponding to the starting points $W(0) = 0$ and $Y(0) = y$. We also write $\ex_W$ for the integration with respect to the law of $W(t)$ alone; equivalently, $\ex_W$ is the conditional expectation with respect to the $\sigma$-algebra generated by the process $Y(t)$.

\begin{proposition}
\label{prop:fourier}
The characteristic function of $X(T_0)$ is given by
\formula{
 \ex^y \bigl( e^{i \xi X(T_0)} \ind_{\{T_0 < T_R\}} \bigr) & = \ph_\xi(y) ,
}
where $\ph_\xi(y)$ is the solution of the ODE~\eqref{eq:ode} discussed in Theorem~\ref{thm:ode}.
\end{proposition}

\begin{proof}
As in Section~4.3 in~\cite{k23}, we fix $\xi > 0$ and we define
\formula{
 \hat X(t) & = \ind_{\{t < T_R\}} \ex_W e^{i \xi X(t)}
}
and
\formula{
 \Phi(t) & = \ind_{\{t < T_R\}} \ph_\xi(Y(t)) .
}
Note that~\eqref{eq:xt} implies that
\formula{
 \hat X(t) & = \ind_{\{t < T_R\}} e^{i \xi B(t)} \ex_W e^{i \xi W(A(t))} = \ind_{\{t < T_R\}} e^{i \xi B(t) - \xi^2 A(t) / 2} .
}
By an application of Itô's lemma and the Itô--Tanaka formula, combined with the properties of $\ph_\xi(y)$ and $\psi(\xi)$, it was proved in~\cite{k23} that with probability one, for every finite $t \in [0, T_R]$ we have
\formula*[eq:ito]{
 \hat X(t) \Phi(t) - \hat X(0) \Phi(0) & = -\psi(\xi) \int_0^t \hat X(s) dL_0(s) \\
 & \qquad + \int_0^t \hat X(s) \ph_\xi'(Y(s)) d\dot Y(s) \\
 & \qquad\qquad + i \xi \int_0^t \hat X(s) \ph_\xi(Y(s)) b(Y(s)) d\dot Y(s) ;
}
for $t \in [0, T_R)$ this is equation~(4.16) in~\cite{k23}, and in the paragraph that follows an extension to $t = T_R$ (when $T_R < \infty$) is proved.

In~\cite{k23}, formula~\eqref{eq:ito} was applied with $t = L_0^{-1}(u) \wedge T_R$ and the process started at $Y(0) = 0$. Here we consider an arbitrary starting point $Y(0) = y \in [0, R)$ and we apply~\eqref{eq:ito} with $t = T_0 \wedge T_R$. In this case $X(0) = 0$, $\hat X(0) = 1$, $\Phi(0) = \ph_\xi(y)$ and $\Phi(T_0 \wedge T_R) = \ph_\xi(0) \ind_{\{T_0 < T_R\}} = \ind_{\{T_0 < T_R\}}$. Furthermore, $L_0(T_0 \wedge T_R) = 0$. Therefore,
\formula{
 \hat X(T_0) - \ph_\xi(y) & = \int_0^{T_0 \wedge T_R} \hat X(s) \ph_\xi'(Y(s)) d\dot Y(s) \\
 & \qquad + i \xi \int_0^{T_0 \wedge T_R} \hat X(s) \ph_\xi(Y(s)) b(Y(s)) d\dot Y(s)
}
with probability $\pr^y$ one (note that $\hat X(T_0) \ind_{\{T_0 < T_R\}} = \hat X(T_0)$). Taking the expectation of both sides eliminates the Itô integrals, and we find that
\formula{
 \ex^y \hat X(T_0) - \ph_\xi(y) & = 0 ,
}
as desired. Note that taking the expectation is allowed because the appropriate integrals are finite by equation~(4.17) in~\cite{k23}.
\end{proof}

Together with Theorem~\ref{thm:main}, Proposition~\ref{prop:fourier} clearly implies Corollary~\ref{cor:main}.

%
%                            ---------- o ----------
%

\section*{}

\subsection*{Acknowledgements}

I thank Jacek Wszoła for inspiring discussions about the subject of the paper. I thank the anonymous referees for their careful review. This research was funded in whole or in part by National Science Centre, Poland, grant number 2023/49/B/ST1/04303. For the purpose of Open Access, the author has applied a CC-BY public copyright licence to any Author Accepted Manuscript (AAM) version arising from this submission.

%
%                            ---------- o ----------
%

%
%                            ---------- o ----------
%

\end{document}